\documentclass{article}

\usepackage{amsmath,amsthm,amssymb,amsfonts}

\usepackage{verbatim}
\usepackage{graphicx}

\usepackage{stmaryrd} 
\usepackage{hyperref}

\usepackage[colorinlistoftodos]{todonotes}

\newtheorem{thm}{Theorem}

\newtheorem{prop}[thm]{Proposition}

\newtheorem{assert}[thm]{Assertion}

\newtheorem{remarks}[thm]{Remark}

\newtheorem{definition}[thm]{Definition}

\newtheorem{exl}[thm]{Example}

\numberwithin{thm}{section}

\newcommand{\adj}{\leftrightarrow}

\newcommand{\adjeq}{\leftrightarroweq}

\DeclareMathOperator{\id}{id}
\DeclareMathOperator{\Fix}{Fix}

\def\Z{{\mathbb Z}}
\def\N{{\mathbb N}}

\def\R{{\mathbb R}}

\begin{document}

\title{Remarks on Fixed Point Assertions in Digital Topology, 6}

\author{Laurence Boxer
\thanks{Department of Computer and Information Sciences, Niagara University, NY 14109, USA
and  \newline
Department of Computer Science and Engineering, State University of New York at Buffalo \newline
email: boxer@niagara.edu
}
}

\date{ }
\maketitle
\begin{abstract}
This paper continues a series discussing flaws in
published assertions concerning fixed points in
digital metric spaces.
\end{abstract}

\section{Introduction}
As stated in~\cite{Bx19}:
\begin{quote}
The topic of fixed points in digital topology has drawn
much attention in recent papers. The quality of
discussion among these papers is uneven; 
while some assertions have been correct and interesting, others have been incorrect, incorrectly proven, or reducible to triviality.
\end{quote}
Here, we continue the work
of~\cite{BxSt19,Bx19,Bx19-3,Bx20,Bx22}, 
discussing many shortcomings in earlier papers 
and offering corrections and improvements.

Quoting and paraphrasing~\cite{Bx22}:
\begin{quote}
Authors of many weak papers concerning 
fixed points in digital topology
seek to obtain results in a ``digital metric space" (see section~\ref{DigMetSp} for its definition).
This seems to be a bad idea. We slightly paraphrase~\cite{Bx20}:
\begin{quote}
\begin{itemize}
    \item Nearly all correct nontrivial published 
    assertions concerning digital
    metric spaces use the metric and do not use the
    adjacency. As a result, the digital metric space
    seems to be an artificial notion.
    \item If $X$ is finite (as in a ``real 
    world" digital image) or the metric $d$ 
          is a common metric such as any $\ell_p$ metric, then 
          $(X,d)$ is uniformly discrete as a topological space, 
          hence not very interesting.
    \item Many published assertions concerning
          digital metric spaces mimic analogues for subsets 
          of Euclidean~$\R^n$. Often, the authors neglect 
          important differences between the topological 
          space $\R^n$ and digital images, resulting in 
          assertions that are incorrect or incorrectly ``proven,"
          trivial, or trivial when restricted to conditions that many regard as 
          essential. E.g., in many cases, functions that
          satisfy fixed point assertions must be constant or fail to be digitally continuous~\cite{BxSt19,Bx19,Bx19-3}.
\end{itemize}
\end{quote}
\end{quote}
Since the publication of~\cite{Bx22}, additional
highly flawed papers rooted in digital metric spaces
have come to our attention. The current paper
discusses shortcomings 
in~\cite{BhardwajEtAl,ChauhanEtAl,JainK,JainR,JainAndKumar,JainEtAl,Kal-J,KamesmariEtAl,KrishnaEtAl,MasmaliEtAl,Shukla}.

\section{Preliminaries}
Much of the material in this section is quoted or
paraphrased from~\cite{Bx20}.

We use $\N$ to represent the natural numbers,
$\Z$ to represent the integers, and $\R$ to represent the reals.

A {\em digital image} is a pair $(X,\kappa)$, where $X \subset \Z^n$ 
for some positive integer $n$, and $\kappa$ is an adjacency relation on $X$. 
Thus, a digital image is a graph.
In order to model the ``real world," we usually take $X$ to be finite,
although there are several papers that consider
infinite digital images. The points of $X$ may be 
thought of as the ``black points" or foreground of a 
binary, monochrome ``digital picture," and the 
points of $\Z^n \setminus X$ as the ``white points"
or background of the digital picture.

\subsection{Adjacencies, 
continuity, fixed point}

In a digital image $(X,\kappa)$, if
$x,y \in X$, we use the notation
$x \adj_{\kappa}y$ to
mean $x$ and $y$ are $\kappa$-adjacent; we may write
$x \adj y$ when $\kappa$ can be understood. 
We write $x \adjeq_{\kappa}y$, or $x \adjeq y$
when $\kappa$ can be understood, to
mean 
$x \adj_{\kappa}y$ or $x=y$.

The most commonly used adjacencies in the study of digital images 
are the $c_u$ adjacencies. These are defined as follows.
\begin{definition}
Let $X \subset \Z^n$. Let $u \in \Z$, $1 \le u \le n$. Let 
$x=(x_1, \ldots, x_n),~y=(y_1,\ldots,y_n) \in X$. Then $x \adj_{c_u} y$ if 
\begin{itemize}
    \item $x \neq y$,
    \item for at most $u$ distinct indices~$i$,
    $|x_i - y_i| = 1$, and
    \item for all indices $j$ such that $|x_j - y_j| \neq 1$ we have $x_j=y_j$.
\end{itemize}
\end{definition}

\begin{definition}
{\rm \cite{Rosenfeld}}
A digital image $(X,\kappa)$ is
{\em $\kappa$-connected}, or just {\em connected} when
$\kappa$ is understood, if given $x,y \in X$ there
is a set $\{x_i\}_{i=0}^n \subset X$ such that
$x=x_0$, $x_i \adj_{\kappa} x_{i+1}$ for
$0 \le i < n$, and $x_n=y$.
\end{definition}

\begin{definition}
{\rm \cite{Rosenfeld, Bx99}}
Let $(X,\kappa)$ and $(Y,\lambda)$ be digital
images. A function $f: X \to Y$ is 
{\em $(\kappa,\lambda)$-continuous}, or
{\em $\kappa$-continuous} if $(X,\kappa)=(Y,\lambda)$, or
{\em digitally continuous} when $\kappa$ and
$\lambda$ are understood, if for every
$\kappa$-connected subset $X'$ of $X$,
$f(X')$ is a $\lambda$-connected subset of $Y$.
\end{definition}

\begin{thm}
{\rm \cite{Bx99}}
A function $f: X \to Y$ between digital images
$(X,\kappa)$ and $(Y,\lambda)$ is
$(\kappa,\lambda)$-continuous if and only if for
every $x,y \in X$, if $x \adj_{\kappa} y$ then
$f(x) \adjeq_{\lambda} f(y)$.
\end{thm}

We use $1_X$ to denote the identity function on $X$, 
and $C(X,\kappa)$ for the set of functions 
$f: X \to X$ that are $\kappa$-continuous.

A {\em fixed point} of a function $f: X \to X$ 
is a point $x \in X$ such that $f(x) = x$. We denote by
$\Fix(f)$ the set of fixed points of $f: X \to X$.

As a convenience, if $x$ is a point in the
domain of a function $f$, we will often
abbreviate ``$f(x)$" as ``$fx$".

\subsection{Digital metric spaces}
\label{DigMetSp}
A {\em digital metric space}~\cite{EgeKaraca15} is a triple
$(X,d,\kappa)$, where $(X,\kappa)$ is a digital image and $d$ is a metric on $X$. The
metric is usually taken to be the Euclidean
metric or some other $\ell_p$ metric; 
alternately, $d$ might be taken to be the
shortest path metric. These are defined
as follows.
\begin{itemize}
    \item Given 
          $x = (x_1, \ldots, x_n) \in \Z^n$,
          $y = (y_1, \ldots, y_n) \in \Z^n$,
          $p > 0$, $d$ is the $\ell_p$ metric
          if \[ d(x,y) =
          \left ( \sum_{i=1}^n
          \mid x_i - y_i \mid ^ p
          \right ) ^ {1/p}. \]
          Note the special cases: if $p=1$ we
          have the {\em Manhattan metric}; if
          $p=2$ we have the 
          {\em Euclidean metric}.
    \item \cite{Han05} If $(X,\kappa)$ is a 
          connected digital image, 
          $d$ is the {\em shortest path metric}
          if for $x,y \in X$, $d(x,y)$ is the 
          length of a shortest
          $\kappa$-path in $X$ from $x$ to $y$.
\end{itemize}

Under conditions in which a digital image 
models a ``real world" image, 
$X$ is finite or $d$ is (usually) an $\ell_p$ metric, so that 
$(X,d)$ is {\em uniformly discrete} as a topological space, i.e., there exists
$\varepsilon > 0$ such that for
$x,y \in X$, $d(x,y) < \varepsilon$ implies
$x=y$.

For an example of a digital metric space
that is not uniformly discrete, see
Example~2.10 of~{\rm \cite{Bx20}}.

We say a sequence $\{x_n\}_{n=0}^{\infty}$ is 
{\em eventually constant} if for some $m>0$, 
$n>m$ implies $x_n=x_m$.
The notions of convergent sequence and complete digital metric space are often trivial, 
e.g., if the digital image is uniformly 
discrete, as noted in the following, a minor 
generalization of results 
of~\cite{Han16,BxSt19}.

\begin{prop}
\label{eventuallyConst}
{\rm \cite{Bx20}}
Let $(X,d)$ be a metric space. 
If $(X,d)$ is uniformly discrete,
then any Cauchy sequence in $X$
is eventually constant, and $(X,d)$ is a complete metric space.
\end{prop}

Let $(X,d)$ be a metric space and
$f: X \to X$. We say $f$ is a 
{\em contraction map}~\cite{EgeKaracaBanach}
if for some $k \in [0,1)$ and
all $x,y \in X$, $d(f(x),f(y)) \le k d(x,y)$.
Such a function is not be confused with
a digital contraction~\cite{Bx94},
a homotopy between an identity map and
a constant function.

\section{On~\cite{BhardwajEtAl}}
This paper has so many typos and unexplained
symbols that it is extremely difficult to
understand. A few examples:
\begin{itemize}
    \item In Definitions~2.1, 2.3, 3.1, and 3.7,
          the ``$\to$" character appears where
          it seems likely that the intended
          character is ``$\times$" to indicate
          a Cartesian product.
    \item In Definition 2.1, it appears the 
          ``$S$" (appearing twice) is intended 
          to be ``$T$", since ``$S$" is
          not defined.
    \item In Definition~3.1, ``$X$" is 
          undefined. It seems likely ``$X$"
          is intended to be ``$\tilde{X}$"
          or something related to the
          latter. If the former, then this
          definition seems to duplicate
          Definition~2.3.
\end{itemize}

The assertions stated as Theorems~4.1 and~4.2 
of~\cite{BhardwajEtAl} promise existence of
unique fixed points. The assertions and their
respective arguments offered as proofs are
very difficult to follow, but one can easily
see that they fail to establish uniqueness of
fixed points.

\section{On~\cite{ChauhanEtAl}}
\subsection{``Theorem" 3.1 of~\cite{ChauhanEtAl}}
The following Definition~\ref{compatibleDef} appears 
in~\cite{ChauhanEtAl}, where it is incorrectly
attributed to~\cite{Bx10}, where it does
not appear. The inspiration for 
Definition~\ref{compatibleDef} may be the rather
different definition of compatible functions
appearing in~\cite{Jungck}.

\begin{definition}
    \label{compatibleDef}
    Suppose $(X,d,\rho)$ is a digital metric space.
    Suppose $P,Q: X \to X$. Then $P$ and $Q$ are {\em compatible}
    if 
    \[ d(PQx, QPx) \le d(Px,Qx) \mbox{ for all } x \in X.
    \]
\end{definition}

The following is stated as Theorem~3.1
of~\cite{ChauhanEtAl}.

\begin{assert}
    \label{Chauhan3.1}
    Let $P$, $Q$, $G$, and $H$ be quadruple
    mappings of a complete digital metric space
    $(X, d, \rho)$ satisfying the following.
    \begin{enumerate}
        \item $G(X) \subset Q(X)$ and
              $H(X) \subset P(X)$.
        \item Let $0 < \alpha < 1$. For all $x,y \in X$, 
              \[ d(x,y) = \alpha \max \left \{ \begin{array}{c}
           d(Gx,Hy), d(Gx,Px), d(Hy,Qy), d(Hx,Qx), \\
           \frac{1}{2}[d(Gx,Qy) + d(Hy,Px)] \end{array}
         \right \}.
        \]
        \item One of $P$, $Q$, $G$, and $H$ is continuous.
        \item The pairs $(P,G)$ and $(Q,H)$ are compatible.
    \end{enumerate}
    Then $P$, $Q$, $G$, and $H$ have a unique common fixed
    point in $X$.
\end{assert}

Flaws in this argument given for this assertion
in~\cite{ChauhanEtAl} include the following.
\begin{itemize}
    \item The first line of the ``proof" reverses the containments
          stated in item 1) of the hypotheses, stating that
          $Q(X) \subset G(X)$ and $P(X) \subset H(X)$.
    \item According to subsequent usage,
          the ``$=$"  in item 2) of the 
          hypothesis should be ``$\le$".
    \item Sequences $\{ \, x_n \, \}$ and $\{ \, y_n \, \}$ are
          constructed, via the rules:
          \begin{itemize}
              \item $x_0$ is an arbitrary point of $X$.
              \item inductively, $y_{2n} = Gx_{2n} = Qx_{2n+1}$ and
                    $y_{2n+1} = Hx_{2n+1}=Px_{2n+2}$.
          \end{itemize}
          Then the following statement appears.
          \[ d(y_{2n}, y_{2n+1}) = d(Gx_{2n},Hx_{2n+1}) \le          
          \]
          \begin{equation}
          \label{ChauhanIneqClaim}
          \alpha \max \left \{ \begin{array}{c}
             d(Px_{2n}, Qx_{2n+1}), d(Px_{2n}, Gx_{2n}),
             d(Qx_{2n+1}, Hx_{2n+1}),  \\
             d(Gx_{2n}, Hx_{2n+1}),
             \frac{1}{2}
             [d(Gx_{2n}, Qx_{2n+1}) + d(Px_{2n}, Hx_{2n+1})]
             \end{array} \right \}.
          \end{equation}
          This statement is not given a 
          justification, and we will show that
          it does not correspond to 
          item 2) of the hypotheses. Direct
          substitution into item 2) of 
          the hypotheses (with the substitution
          of ``$\le$" for ``=") yields the
          following.
          \begin{equation} 
          \label{ChauhanIneq}
          d(y_{2n},y_{2n+1}) \le 
          \alpha \max \left \{ 
          \begin{array}{c}
          d(Gy_{2n}, Hy_{2n+1}),
          d(Gy_{2n},Py_{2n}), \\
          d(Hy_{2n+1}, Qy_{2n+1}), 
          d(Hy_{2n},Qy_{2n}), \\
          \frac{1}{2}
          [d(Gy_{2n}, Qy_{2n+1}) +
                d(Hy_{2n+1},Py_{2n})]
          \end{array} \right \}
          \end{equation}
          The expression $Gy_{2n}$ appears 3
          times in~(\ref{ChauhanIneq}). No
          provision occurs in~\cite{ChauhanEtAl}
          for converting this expression into
          an expression of the sequence
          $\{ \, x_n \}$, as
          in~(\ref{ChauhanIneqClaim}).
\end{itemize}
We must conclude that Assertion~\ref{Chauhan3.1} is unproven.

\subsection{Example 3.2 of~\cite{ChauhanEtAl}}
This example is based on $[0,\infty)$, which
the authors want to consider as a digital
metric space. It is not a digital metric
space, as $[0,\infty)$ is not a subset of
$\Z^n$ for any~$n$. Nor is a conclusion stated for
this example.

\section{On~\cite{JainK,JainR}}
The papers~\cite{JainK,JainR} introduce, for
$f,g: X \to X$ on a digital metric space
$(X,d,\kappa)$, notions
of {\em compatible of type K} and
and {\em compatible of type R}, claiming that
they yield fixed point results. In
section~\ref{compatVariants} we discuss 
how these notions are related to other notions 
of compatibility that have appeared 
in the literature. In section~\ref{JainKAsserts}
we show flaws in the
fixed point assertions of~\cite{JainK,JainR}.

\subsection{Variants on compatibility}
\label{compatVariants}
In this section, we consider sequences 
$\{ \, x_n \, \} \subset X$ such that
\begin{equation}
\label{equiconverge}
    lim_{n \to \infty} fx_n =
       lim_{n \to \infty} gx_n = t \in X
\end{equation}

\begin{definition}
    \label{compatible}
    {\rm \cite{DalalEtAl}}
    Suppose $f$ and $g$ are self-functions on 
    a metric space $(X,d)$. If every sequence
    $\{ \, x_n \, \} \subset X$ 
    satisfying~{\rm (\ref{equiconverge})}
    also satisfies
    \[ lim_{n \to \infty} d(f(gx_n), g(fx_n))=0, 
    \]
    then $f$ and $g$ are {\em compatible}.
\end{definition}

\begin{definition}
{\rm \cite{JainK}}
\label{typeR}
    Let $(X,d)$ be a metric space.
    Let $f,g: X \to X$. We say $f$ and $g$ 
    are {\em digitally compatible of type K} if
    for every infinite sequence 
    $\{ \, x_n \, \} \subset X$ 
    satisfying~{\rm (\ref{equiconverge})} we have
\begin{equation}
\label{JainKEq}
\lim_{n \to \infty} d(ffx_n, gt) = 0
     ~~~ \mbox{and} ~~~
    \lim_{n \to \infty} d(ggx_n,ft) = 0.
\end{equation}
\end{definition}

\begin{definition}
{\rm \cite{JainR}}
\label{typeR}
    Let $(X,d)$ be a metric space.
    Let $f,g: X \to X$. We say $f$ and $g$ 
    are {\em digitally compatible of type R} if
    for every infinite sequence 
    $\{ \, x_n \, \} \subset X$ 
    satisfying~{\rm (\ref{equiconverge})} we have
    \[
    \lim_{n \to \infty} d(fgx_n, gfx_n) = 0
       = \lim_{n \to \infty} d(ffx_n, ggx_n).
    \]
\end{definition}

We have the following.

\begin{prop}
    \label{compatAndCompatK}
     Suppose S and T are self-functions on 
    a metric space $(X,d)$ that is
    uniformly discrete. Then
    $S$ and $T$ are compatible of type K
    if and only if $S$ and $T$ are compatible.
\end{prop}

\begin{proof}
   Let $\{ \, x_n \, \} \subset X$ be a sequence
    satisfying~{\rm (\ref{equiconverge})} for the
    functions $S,T$. The uniform discreteness 
    hypothesis implies for almost all~$n$
    \[ Sx_n = t = Tx_n, ~~~ SSx_n = St, ~~\mbox{ and } ~~ TTx_n = Tt.
    \]
If $S$ and $T$ are compatible, 
    \[ d(SSx_n,Tt) = d(STx_n,TSx_n) = 0       
        = d(TSx_n,STx_n) = d(TTx_n, St)
    \]
    so $S$ and $T$ are compatible of
    type K.

    Suppose $S$ and $T$ are compatible of type K. Then
    for almost all~$n$, 
    \[ 0 = d(SSx_n,Tt) = d(STx_n, TSx_n)
    \]
    so $S$ and $T$ are compatible.
\end{proof}

\begin{prop}
\label{compatAndCompatR}
     Suppose S and T are self-functions on 
    a metric space $(X,d)$. If
    $S$ and $T$ are compatible of type R, then
    $S$ and $T$ are compatible. The converse
    is true if $(X,d)$ is uniformly discrete.
\end{prop}

\begin{proof}
It is clear from Definitions~\ref{compatible}
and~\ref{typeR} that a type R pair is a 
compatible pair.

Suppose $S$ and $T$ are compatible and
$(X,d)$ is uniformly discrete. Let 
$\{ \, x_n \, \} \subset X$ 
    satisfy~(\ref{equiconverge}).
\begin{itemize}
    \item By Definition~\ref{compatible}, 
      $\lim_{n \to \infty} d(fgx_n, gfx_n) = 0$.
    \item Since $(X,d)$ is uniformly discrete,
          for almost all~$n$ we have
          $fx_n = gx_n = t$. Therefore, 
          for almost all~$n$,
          \[ ffx_n = fgx_n = gfx_n = ggx_n.
          \]
\end{itemize}
Therefore, $S$ and $T$ are compatible of type R.
\end{proof}

Several variants of compatible functions
have been defined in the literature. These 
include compatible of type A~\cite{DalalEtAl},
compatible of type B~\cite{EgeEtAl},
compatible of type C~\cite{EgeEtAl}, and
compatible of type P~\cite{DalalEtAl}. 

\begin{thm}
\label{manyEquivs}
Let $(X,d)$ be a metric space
that is uniformly discrete.
Let $S,T: X \to X$. The following are equivalent.
\begin{itemize}
    \item $S$ and $T$ are compatible.
    \item $S$ and $T$ are compatible of type A.
    \item $S$ and $T$ are compatible of type B.
    \item $S$ and $T$ are compatible of type C.
    \item $S$ and $T$ are compatible of type K.
    \item $S$ and $T$ are compatible of type P.
    \item $S$ and $T$ are compatible of type R.
\end{itemize} 
\end{thm}

\begin{proof}
The equivalence of compatible, compatible of 
type A, compatible of type B,
compatible of type C, and
compatible of type P, is shown in
Theorem~3.9 of~\cite{Bx20}. 
The equivalence of
compatible and compatible of type K is shown 
in Proposition~\ref{compatAndCompatK}.
The equivalence of
compatible and compatible of type R is shown 
in Proposition~\ref{compatAndCompatR}.
\end{proof}

\subsection{On the fixed point assertion
of~\cite{JainK}}
\label{JainKAsserts}
The following is stated as Theorem~4.1
of~\cite{JainK}.

\begin{assert}
    \label{JainK4.1}
    Let $A,B,S,T: X \to X$ where
    $(X,d,\kappa)$ is a digital metric space
    such that
    \begin{enumerate}
        \item $S(X) \subset B(X)$ and
              $T(X) \subset A(X)$.
        \item For all $x,y \in X$ and some
              $\alpha \in (0,1)$,
              \[ \begin{array}{l}
              d(Sx,Ty) = \\
              \alpha \max \{ \,
              d(Ax,By), d(Ax,Sx), d(By,Ty),
              d(Sx,By), d(Ax,Ty) \, \}.
              \end{array}
              \]
              {\em [Probably, the comparison
              operator in the latter statement
              should be ``$\le$" rather than
              ``=".]}
       \item Pairs $(A, S)$ and $(B, T)$ are reciprocally continuous.
       \item $(A,S)$ and $(B,T)$ are pairs of
             compatible of type K functions.
    \end{enumerate}
    Then $A,B,S$, and $T$ have a unique common fixed point in $X$.
\end{assert}

We note the following flaw in the argument 
offered as proof of this assertion. A sequence
$\{ \, y_n \, \}$ is constructed. It is claimed
that this is a Cauchy sequence, ``From the proof
of [{\em what the current paper will 
refer to as}~\cite{JainEtAl}]".
We find that~\cite{JainEtAl} is listed 
in~\cite{JainK} as submitted to the
{\em Journal of Mathematical Imaging 
and Vision}. At the current writing, 
4 years after the publication of~\cite{JainK},
neither a search of the {\em JMIV} website 
nor a general Google search succeeds in 
locating~\cite{JainEtAl}.

We conclude that 
Assertion~\ref{JainK4.1} is unproven.

\subsection{On fixed point assertions
of~\cite{JainR}}
\begin{prop}
\label{JainR3.2}
Let $f,g: X \to X$ be compatible of type R on
a uniformly discrete digital metric space 
$(X,d, \kappa)$. If 
\begin{equation}
\label{JainR3.2Eq1}
    ft=gt \mbox{ for some $t \in X$},
\end{equation}
then $fgt = fft = ggt = gft$.
\end{prop}

\begin{proof}
We remark that this assertion is a modified
version of Proposition~3.2 of~\cite{JainR};
the latter appeared in~\cite{JainR} with neither 
proof nor citation. The modification we have
made is inclusion of the assumption of
uniform discreteness.

Let $x_n = t$ for all~$n$. From
    Definition~\ref{typeR},
    $\lim_{n \to \infty} d(fgx_n, gfx_n) = 0$
    and
     $\lim_{n \to \infty} d(ffx_n, ggx_n) = 0$.
    Since $(X,d)$ is uniformly discrete, for
    almost all~$n$,
    $fgx_n = gfx_n$ and $ffx_n = ggx_n$,
    hence $fgt = gft$ and $fft = ggt$.
    We complete the proof by observing
    that~(\ref{JainR3.2Eq1}) implies $fgt = fft$.
\end{proof}

\begin{prop}
    \label{JainR3.3}
    Let $f,g: X \to X$ be compatible of type R on
a uniformly discrete digital metric space 
$(X,d, \kappa)$. If $\{ \, x_n \, \}$ is a 
sequence in~$X$ satisfying~(\ref{equiconverge}), then 
\begin{enumerate}
    \item $ft=gt$;
    \item $\lim_{n \to \infty} gfx_n = ft$;
    \item  $\lim_{n \to \infty} fgx_n = gt$; and
    \item $fgt = gft$.
\end{enumerate}
\end{prop}

\begin{proof}
    We remark that this proposition is a
    modified version of Proposition~3.3
    of~\cite{JainR}, which appears there
    with neither proof nor citation. Our
    modification is to replace Jain's
    assumptions of digital continuity (which
    Jain may have confused with metric
    continuity, since it seems unlikely that one
    can easily show that
    the desired outcome follows from digital
    continuity) with the assumption of 
    ``uniformly discrete".

\begin{enumerate}
    \item From (\ref{equiconverge}) and the assumption
    of uniform discreteness, it follows that
    for almost all~$n$, $fx_n=gx_n=t$.
    From compatibility of type R, it follows
    that 
    \begin{equation}
    \label{JainR3.3Eq}
        0 = lim_{n \to \infty} d(fgx_n, gfx_n) =
           d(ft,gt).
    \end{equation}
    Thus, $ft = gt$.
    \item By the uniform discreteness property,
          $\lim_{n \to \infty} gfx_n = gt = ft$.
    \item Similarly,
          $\lim_{n \to \infty} fgx_n = ft = gt$.
    \item $fgt = gft$ by Proposition~\ref{JainR3.2}.
\end{enumerate}
\end{proof}

The following is stated as Theorem~4.1
of~\cite{JainR}.

\begin{assert}
    \label{JainR4.1}
    Let $(X,d,\kappa)$ be a digital metric space.
    Let $A,B,S,T: X \to X$ such that
    \begin{enumerate}
        \item $S(X) \subset B(X)$ and
              $T(X) \subset A(X)$;
        \item for $0 < \alpha < 1$ and all
              $x,y \in X$,
              \[ d(Sx,Ty) \le 
              \alpha \max \left \{
              \begin{array}{c}
              d(Ax,By), d(Ax,Sx), d(By,Ty), \\
              d(Sx,By), d(Ax,Ty)
              \end{array} \right \};
              \]
        \item one of the functions
              $A,B,S,T$ is continuous; and
        \item $(A,S)$ and $(B,T)$ are pairs of
              functions that are compatible of
              type R.
    \end{enumerate}
    Then $A,B,S$, and $T$ have a common
    fixed point in~$X$.
\end{assert}

The argument offered in~\cite{JainR} as proof of
Assertion~\ref{JainR4.1} is flawed as
follows. Sequences 
$\{ \, y_n \, \} \subset X$ and
$\{ \, d_n = d(y_n, y_{n+1}) \, \}$ are
constructed, and it is shown that
for all~$n$,
\begin{equation}
\label{JainR4.1eq1}
   d_{2n} \le \alpha \max \{ \,
          d_{2n-1}, d_{2n}, d_{2n-1} + d_{2n}
          \, \}
\end{equation}
and $d_{2n} \le d_{2n - 1}$, i.e.,
\begin{equation}
    \label{JainR4.1eq2}
    d(y_{2n}, y_{2n+1}) \le d(y_{2n-1}, y_{2n}).
\end{equation}
Note for inequality~(\ref{JainR4.1eq1}) 
the index on the left side is even,
and in~(\ref{JainR4.1eq2}), the smaller index
on the left side is even.
In an apparent attempt to obtain a
geometric series from the 
inequalities~(\ref{JainR4.1eq1}) 
and~(\ref{JainR4.1eq2}), it is claimed these
inequalities imply that
for $m,n \in \N$ with $m > n$,
\begin{equation}
 \label{JainR4.1eq3}
d(y_m, y_n) \le \alpha d(y_m,y_{m-1}) +
   \ldots + \alpha d(y_{n+1},y_n).
\end{equation}

But since there is no analogue of either of the 
inequalities~(\ref{JainR4.1eq1}) 
and~(\ref{JainR4.1eq2}) in which the
index on the left side of~(\ref{JainR4.1eq1})
is odd or the smaller index on the left
side of~(\ref{JainR4.1eq2}) 
is odd, it is not clear 
that one can
obtain~(\ref{JainR4.1eq3}).

Thus, we must consider Assertion~\ref{JainR4.1}
as unproven.

\section{On~\cite{JainAndKumar}}
We will show that the assertions of~\cite{JainAndKumar}
are, for the most part, trivial in that their
hypotheses are impossible.

\subsection{Some definitions and elementary properties}
\begin{definition}
    \label{PsiDef}
    {\rm \cite{JainAndKumar}}
    $\Psi$ is the set of nondecreasing functions
    $\psi: [0,\infty) \to [0,\infty)$ such that
  $\psi(0)=0$.
\end{definition}

\begin{remarks}
    \label{PsiFuncIs0}
Remark 2.13(1) of~\cite{JainAndKumar} says if 
$\psi: [0,\infty) \to [0,\infty)$
is nondecreasing such that
$\sum_{n=1}^{\infty} \psi^n(t) < \infty$ for all $t \ge 0$, 
then $\psi \in \Psi$. This turns 
out to be trivial, because
such a function must be constant, 
with value $0$.
\end{remarks}

\begin{proof}
    We have $\psi(x) \ge 0$ for all $x \in [0,\infty)$.
    Suppose there exists $t > 0$ such that
    $\psi(t) = q > 0$; let this be the base
    case of an induction to show $\psi^n(t) \ge q$
    for all $n \in \N$.

    Suppose we have $\psi^n(t) \ge q$ for $n \le k$.
    Since $\psi$ is nondecreasing,
    \[ \psi^{k+1}(t) = \psi(\psi^k(t)) \ge \psi(q) \ge q, \]
    which completes the induction.

    Therefore, $\sum_{i=1}^n \psi^n (t) \ge nq$ for all
    $n \in \N$, contrary to the assumption that
    $\sum_{n=1}^{\infty} \psi^n(t) < \infty$. The
    contradiction establishes the assertion.
\end{proof}

\subsection{Remark 3.2 of~\cite{JainAndKumar}}
\begin{definition}
{\rm \cite{JainAndKumar}}
\label{APexpansive}
Let $S: X \to X$ for a digital metric space
$(X,d,\kappa)$. Let
$\alpha: X \times X \to [0,\infty)$ and
$\psi \in \Psi$ such that for all $x,y \in X$,
\[ \psi(d(Sx,Sy)) \ge \alpha(x,y) d(x,y).
\]
Then $S$ is a {\em digital $\alpha - \psi$ expansive} mapping.
\end{definition}

A self-map $S$ on a metric
space $(X,d)$ is 
{\em expansive} if for all
$x,y \in X$,
$d(Sx,Sy) \ge d(x,y)$.

Remark 3.2 of~\cite{JainAndKumar} 
claims the following.
\begin{assert}
\label{JKremark3.2}
    Any expansive mapping is a digital
    $\alpha - \psi$ expansive mapping with
    $\alpha(x,y) = 1$ for all $x,y \in X$
    and $\psi(t) = kt$ for $0 < k < 1$.
\end{assert}

\begin{remarks}
Assertion~\ref{JKremark3.2} is
not generally true.
It is clear that $\id_X$ is an expansive
mapping. Let $S = \id_X$. Under the assumptions
of Assertion~\ref{JKremark3.2}, 
we would have from 
Definition~\ref{APexpansive}
the inequality
\[ k d(x,y) \ge d(x,y),~~~~~
\mbox{which is false for } x \neq y.
\]
\end{remarks}

\subsection{``Theorems" 3.4 and 3.5
of~\cite{JainAndKumar}}
\begin{definition}
    {\rm \cite{JainAndKumar}}
    Let $S: X \to X$ and 
    $\alpha: X \times X \to [0,\infty)$. We
    say $S$ is $\alpha$-admissible if
    for all $x,y \in X$,
    \[ \alpha(x,y) \ge 1 \Rightarrow 
        \alpha(Sx,Sy) \ge 1.
    \]
\end{definition}

The following is stated as Theorem~3.4
of~\cite{JainAndKumar}.
\begin{assert}
    \label{JKthm3.4}
    Let $(X,d,\kappa)$ be a 
    complete digital metric space.
    Let $S: X \to X$ be a bijective, digital
    $\alpha - \psi$-expansive mapping such that
    \begin{itemize}
        \item $S^{-1}$ is $\alpha$-admissible;
        \item for some $x_0 \in X$,
              $\alpha(x_0,S^{-1}(x_0)) \ge 1$;
              and
        \item $S \in C(X,\kappa)$.
    \end{itemize}
    Then $S$ has a fixed point in $X$.
\end{assert}

The following is stated as Theorem~3.5
of~\cite{JainAndKumar}.

\begin{assert}
    \label{JK3.5}
    Suppose we replace the continuity
    assumption in Assertion~\ref{JKthm3.4}
    by the assumption that
    \[
       \label{JKthm3.5eq}
         \mbox{if  $\{ \, x_n \, \} \subset X$
         and  $\lim_{n \to \infty} x_n =x$,
         then $\alpha(S^{-1}x_n, S^{-1}x) \ge 1$
         for all~$n$}.
    \]
    Then $S$ has a fixed point in $X$.
\end{assert}

\begin{remarks}
Both of 
Assertion~\ref{JKthm3.4} and
Assertion~\ref{JK3.5} are false, as
shown by the example
\[ S: \Z \to \Z \mbox{ given by }
S(x) = x+1;~~
\alpha(x,y) = 1.
\]
\end{remarks}

\subsection{Example 3.7}
Example~3.7 of~\cite{JainAndKumar}
wishes to consider a digital metric space
$(X,d,\kappa)$ and a function $T:X \to X$
defined by
\[ T(x) = \left \{ \begin{array}{ll}
   2x - \frac{11}{6} & \mbox{if } x > 1; \\
   \frac{x}{6} & \mbox{if } x \le 1.
   \end{array} \right .
\]
We are not told what $X$ is. From the 
definition of~$T$, we must have
$X \subset \Z$. But~$T$
does not appear to be 
integer-valued, e.g., if $1 \in X$ 
then $T(1) \not \in \Z$.

\section{On~\cite{Kal-J}}
We consider Theorems~3.1 
and~3.2 of~\cite{Kal-J}, stated below
as Theorem~\ref{KJthm3.1} and Assertion~\ref{KJthm3.2}, 
respectively.
We show that the former
reduces to triviality and
the latter is false.

The following is stated as Theorem~3.1 of~\cite{Kal-J}.

\begin{thm}
\label{KJthm3.1}
Let $(X,d,\kappa)$ be a complete digital metric space
and suppose $T: (X,d,\kappa) \to (X,d,\kappa)$ satisfies
$d(Tx,Ty) \le \psi(d(x,y))$ for all $x,y \in X$, where
$\psi: [0,\infty) \to [0,\infty)$ is monotone
non-decreasing and satisfies 
\begin{equation}
    \label{psiRaised}
    \lim_{n \to \infty} \psi^n(t) = 0 \mbox{ for all $t > 0$}.
\end{equation}
Then $T$ has a unique fixed point in $(X,d,\kappa)$.
\end{thm}

We show that
Theorem~\ref{KJthm3.1}
reduces to triviality.

\begin{prop}
    \label{psiTrivial}
    Let $\psi$ be as in Theorem~\ref{KJthm3.1}. Then 
    \begin{itemize}
    \item $\psi$ is the
          constant function
          with value 0.
    \item $T$ is a constant
          function.
    \end{itemize}
\end{prop}

\begin{proof}
    Suppose we have some $t_0$ for which $\psi(t_0) > 0$. Then a simple induction
    shows that $\psi^{n+1}(t_0) \ge \psi(t_0) > 0$ for all $n \in \N$, contrary
    to~(\ref{psiRaised}).
    Thus $\psi$ must be the
          constant function
          with value 0.

    It follows that
    $d(Tx,Ty) = 0$ for
    all $x,y \in X$, so
    $T$ is a constant function.
\end{proof}

Theorem~3.2 of~\cite{Kal-J} depends on
the following.

\begin{definition}
    {\rm \cite{DolhareEtAl}}
    \label{weaklyUniformlyStrict}
    Let $(X,d,\kappa)$ be a digital metric
    space. Then $T: X \to X$ is a
    {\em weakly uniformly strict digital 
    contraction} if given $\varepsilon > 0$
    there exists $\delta > 0$ such that
    $\varepsilon \le d(x,y) < \varepsilon + \delta$
    implies $d(Tx,Ty) < \varepsilon$ 
    for all $x,y \in X$.
\end{definition}

But Definition~\ref{weaklyUniformlyStrict}
turns out to be a triviality in many
important cases, as shown by the
following.

\begin{prop}
\label{negateKJ3.2}
Let $(X,d)$ be a finite metric 
space. Then every weakly uniformly 
strict digital contraction on~$X$
is a contraction map.
\end{prop}

\begin{proof}
Let $x \neq y$ in~$X$. Let
$\varepsilon = d(x,y) > 0$.
Let $T: X \to X$ be a weakly 
uniformly strict digital contraction 
on~$X$. It follows from
Definition~\ref{weaklyUniformlyStrict}
that $d(Tx,Ty) < \varepsilon = d(x,y)$.

Since $X$ is finite,
there exists $k \in (0,1)$ such
that for all $x,y \in X$,
$d(Tx,Ty) \le kd(x,y)$.
\end{proof}

The following is stated as
Theorem~3.2 of~\cite{Kal-J}. 

\begin{assert}
    \label{KJthm3.2}
    Let $(X,d,\kappa)$ be a complete 
    digital metric space.
    Let $T: (X,d,\kappa) \to (X,d,\kappa)$ 
    be a weakly uniformly
strict digital contraction mapping. Then $T$ has a unique fixed point $z$.
Moreover, for any $x \in X$, $lim_{n \to \infty} T^nx = z$. 
\end{assert}

\begin{remarks}
Assertion~\ref{KJthm3.2} is trivial
if $X$ is finite and $c_1$-connected
and $d$ is an $\ell_p$ metric or
the shortest path metric.
\end{remarks}

\begin{proof}
    By Proposition~\ref{negateKJ3.2},
    $T$ is a contraction map.
    Therefore~\cite{BxSt19}, our
    hypotheses imply $T$
    is a constant map.
\end{proof}

\section{On~\cite{KamesmariEtAl}}
The main result of~\cite{KamesmariEtAl} is
Theorem~\ref{KamThm}, below.
This theorem depends on the notions
of {\em $1$-chainable} and {\em uniformly local
contractive mapping}; it seems unnecessary to
define these here.

\begin{thm}
    \label{KamThm} 
    Let $(X,d,\ell)$ be a $1$-chainable
    complete digital metric space. Let
    $T: X \to X$ be a $(1,\ell)$-uniformly locally contractive mapping. Then $T$ has
    a unique fixed point in $X$.
\end{thm}

This assertion is correct. However, its publication is unfortunate, for the following reasons.
\begin{itemize}
\item The assertion turns out to be trivial,
since such a map must be constant; indeed,
the latter is shown in~\cite{KamesmariEtAl}.
\item Theorem~\ref{KamThm} duplicates 
      Theorem~5.1 of the 
      earlier~\cite{HossainEtAl} (a result
      later shown to be trivial 
      in~\cite{BxSt19}).
\end{itemize}

\section{On~\cite{KrishnaEtAl, KrishnaAndTatajee}}
\subsection{Geraghty contraction}
The papers~\cite{KrishnaEtAl,KrishnaAndTatajee}
focus on digital Geraghty contraction maps.

Unfortunately, these papers use the symbol
$S$ both to represent a single function
and a certain set of functions. This leads
to an unfortunate simultaneous use of 
both interpretations of this symbol.
We will try to unravel this confusion by using
the following, with $G$ as the symbol for the
set of functions discussed.

\begin{definition}
    {\rm \cite{KrishnaAndTatajee}}
    $G = \{ \, \beta: [0,\infty) \to [0,1)
    \mid \beta(t_n) \to 1 \Rightarrow 
    t_n \to 0 \, \}$.
\end{definition}

\begin{definition}
     {\rm \cite{KrishnaAndTatajee}}
     \label{GeraghtyDef}
     Let $(X,d,\kappa)$ be a digital metric
     space. The function $f: X \to X$ is
     a digital Geraghty contraction map if 
     there exists $\beta \in G$ such that
     \[ d(f(x),f(y)) \le \beta(d(x,y)) d(x,y)
        \mbox{ for all } x,y \in X.     
     \]
\end{definition}

\begin{remarks}
    It is observed in~{\rm \cite{KrishnaAndTatajee}}
    that a digital contraction map is a 
    digital Geraghty contraction map, but
    the converse is not true. However, the
    converse is true for finite $X$, since
    in this case we can replace 
    $\beta(d(x,y))$ by 
    \[\beta'(d(x,y)) = \max \{ \, \beta(d(x,y))
       \mid x,y \in X \, \} < 1,
     \]
     since
     \[ d(fx,fy) \le \beta(d(x,y)) d(x,y) \le 
        \beta'(d(x,y)) d(x,y).
     \]
\end{remarks}

The following shows an important case in which
Definition~\ref{GeraghtyDef} reduces to
triviality.

\begin{prop}
\label{GeraghtyTriv}
    Suppose $(X,d,\kappa)$ is a digital 
    metric space and $d$ is an $\ell_p$ metric
    or the shortest path metric. If
    $X$ is $c_1$-connected (whether or not
    $\kappa$ is $c_1$) then a
    digital Geraghty contraction map must
    be a constant function.
\end{prop}

\begin{proof}
    Let $x \adj_{c_1} y$ in $X$. Then
    $d(x,y) = 1$. From
    Definition~\ref{GeraghtyDef},
    \[ d(fx,fy) < \beta(1) < 1, ~~~\mbox{ so }~~~
    d(fx,fy) = 0.
    \]
    Since $(X,c_1)$ is connected,
    the assertion follows.
\end{proof}

\subsection{Theorem of~\cite{KrishnaAndTatajee}}
\begin{thm}
    {\rm \cite{KrishnaAndTatajee}}
    \label{KandT}
    Let $(X,d,\kappa)$ be a complete digital
    metric space, where $d$ is the Euclidean
    metric. Let $f: X \to X$ be a
    digital Geraghty contraction map. Then
    $f$ has a fixed point in~$X$.
\end{thm}

The assertion of this theorem is correct.
Here, we discuss flaws in the argument offered
for its proof in~\cite{KrishnaAndTatajee},
resulting in a much longer argument than
necessary.
\begin{itemize}
    \item Given $x_0 \in X$ and the sequence 
$\{ \, x_n = f(x_{n-1}) \, \}$, it is shown
that $\lim_{n \to \infty} d(x_n,x_{n+1}) = 0$ is a
decreasing sequence. Since $(X,d)$ is a complete metric
space, this suffices to conclude there exists
$u \in X$ such that $x_n \to u$.

However, instead the authors seek
to show the sequence is Cauchy by obtaining
a contradiction to the assumption that it isn't,
leading to a choice of $\varepsilon > 0$ and
a subsequence of $\{ \, x_n = f(x_{n-1}) \, \}$
with members with arbitrarily large indices
$u, v$ such that
\[ \varepsilon \le d(x_u, x_v) < \varepsilon.
\]
Rather than recognize that this gives the
desired contradiction, the authors proceed with
several more paragraphs before concluding
that they have the desired contradiction.
\item The authors neglect to complete the
      proof, failing to show the existence
      of a unique fixed point. This can be done
      as follows. We have
      \[ u = \lim_{n \to \infty} x_n =
         \lim_{n \to \infty} f(x_n).
      \]
      Since $d$ is the Euclidean metric,
      $(X,d)$ is uniformly discrete, so for almost
      all~$n$,
      \[ u = x_n = x_{n+1} = f(x_n) = f(u).
      \]
      Thus, $u$ is a fixed point.
    
      The uniqueness of $u$ as a fixed point
          of $f$ is shown as follows.
          Suppose $u'$ is a fixed point.
          Then
          \[ d(u,u') = d(fu, fu') \le 
             \beta(u,u') d(u,u')
          \]
          which implies $\beta(u,u') d(u,u')=0$
          and thus $d(u,u')=0$, so $u = u'$.  
\end{itemize}

\begin{remarks}
    In light of Proposition~\ref{GeraghtyTriv}, we
    see that there are many cases
    for which Theorem~\ref{KandT} is
    trivial.
\end{remarks}

\subsection{On~\cite{KrishnaEtAl}}
The paper~\cite{KrishnaEtAl} is concerned with
pairs of Geraghty contraction maps in digital metric spaces.

\begin{definition}
 {\rm \cite{KrishnaEtAl}}
 \label{GeraghtyPair}
 Let $(X,d,\kappa)$ be a digital metric space.
 Let $S,T: X \to X$, If there exists 
 $\beta \in G$ such that for all $x,y \in X$,
 \[ d(Sx,Sy) \le \beta(d(Tx,Ty)) d(Tx,Ty),
 \]
 then $(T,S)$ is a pair of Geraghty contraction maps.
\end{definition}

We note an important case in which
Definition~\ref{GeraghtyPair} reduces
to triviality.

\begin{prop}
\label{Gconst}
    Let $(X,d,c_1)$ be a digital metric space
    and let $(T,S)$ be a pair of Geraghty contraction maps on $X$. If 
    \begin{itemize} 
    \item $d$ is any $\ell_p$ 
    metric or the shortest path metric;
    \item $T \in C(S,c_1)$; and
    \item $(X,c_1)$ is connected,
    \end{itemize}
    then $S$ is a constant function.
\end{prop}

\begin{proof}
    Let $x \adj_{c_1} y$ in $X$. By
    Definition~\ref{GeraghtyPair},
    \[ d(Sx,Sy) \le \beta(d(Tx,Ty)) (d(Tx,Ty))
       < d(Tx,Ty).
    \]
    Since $Tx \adjeq_{c_1} Ty$, we have
    $d(Sx,Sy) < 1$, so $d(Sx,Sy) = 0$. Since
    $(X,c_1)$ is connected, $S$ must be constant.
\end{proof}

The following is stated as Theorem~2.2
of~\cite{KrishnaEtAl}.

\begin{assert}
    {\rm \cite{KrishnaEtAl}}
    \label{KrishEtAlthm}
    Let $(T,S)$ be a pair of Geraghty contraction maps on $X$, where
    $(X,d,\kappa)$ is a digital metric space.
    Suppose
    \begin{itemize}
        \item $S(X) \subset T(X)$;
        \item $T$ is continuous; and
        \item $S$ and $T$ commute.
    \end{itemize}
    Then $T$ and $S$ have a common fixed point.
\end{assert}

The argument given as proof of this assertion is
flawed as follows. A sequence 
$\{ \, x_n \, \}$ is formed in $X$ such that
$\{ \, d(Tx_{n+1},Tx_n) \, \}$ is a
decreasing sequence. It is derived that
\[ \frac{d(Tx_{n+2},Tx_{n+1})}{d(Tx_{n+1},x_n)}
   \le \beta(d(Tx_{n+1},Tx_n)) < 1.
\]
It is then claimed that the latter implies
\[ \lim_{n \to \infty} 
   \frac{d(Tx_{n+2},Tx_{n+1})}{d(Tx_{n+1},x_n)} 
   = 1.
\]
No justification is given for this claim.
Therefore, Assertion~\ref{KrishEtAlthm} must
be regarded as unproven.

We note in the following that there
are important cases for which 
Assertion~\ref{KrishEtAlthm} reduces 
to triviality.

\begin{exl}
 Let $(T,S)$ be a pair of Geraghty contraction maps on $X$, where
    $(X,d,c_1)$ is a digital metric space.
    Suppose
    \begin{itemize}
        \item $T \in C(X,c_1)$;
         \item $d$ is any $\ell_p$ 
    metric or the shortest path metric;
        \item $(X,c_1)$ is connected; and
        \item $S$ and $T$ commute.
    \end{itemize}
    Then $S$ is a constant function, and
    $S$ and $T$ have a common fixed point.
\end{exl}

\begin{proof}
    By Proposition~\ref{Gconst}, $S$ is
    constant; say, $S(x) = x_0$ for all
    $x \in X$. Since $S$ and $T$ commute,
    \[ Tx_0 = TSx_0 = STx_0 = x_0.
    \]
\end{proof}

\section{On~\cite{MasmaliEtAl}}
\subsection{Admissible functions}

\begin{definition}
    {\rm \cite{SriEtAl}}
    \label{MasmaliEtAl-1functAdmiss}
    Consider functions $T: X \to X$ and
    $\alpha: X \times X \to [0,\infty)$.
    We say $T$ is {\em $\alpha$-admissible} if
    \[ x,y \in X, ~ \alpha(x,y) \ge 1 ~ \Rightarrow ~
       \alpha(Tx,Ty) \ge 1.
    \]
\end{definition}

\begin{definition}
    {\rm \cite{MasmaliEtAl}}
    \label{MasmaliDefAdmissible}
    Let $S,T: X \to X$, $\alpha: X \times X \to [0, \infty)$.
    We say $S$ is $\alpha - \beta$-admissible with respect to $T$
    if for all $x,y \in X$ we have $\alpha(Tx,Ty) \ge 1$ and $\beta(Tx,Ty) \ge 1$
    implies $\alpha(Sx,Sy) \ge 1$ and $\beta(Sx,Sy) \ge 1$.
\end{definition}
Presumably, $\beta: X \times X \to [0, \infty)$ also,
but this is not stated in the definition quoted above.

\subsection{``Theorem" 3.1 of~\cite{MasmaliEtAl}}
Let 
\[ \Phi = \left \{ \begin{array}{c} 
          \varphi: [0, \infty) \to [0, \infty) 
          \mid \varphi \mbox{ is increasing, }
           \\ t> 0 \Rightarrow \varphi(t) <t, 
           \mbox{ and }
           \varphi(0) = 0 
           \end{array} \right \}
\]

\begin{definition}
    {\rm \cite{MasmaliEtAl}}
    Let $(X,d,\rho)$ be a complete digital metric space. Let
    $T: X \to X$, $\alpha, \beta: [0, \infty) \to [0, \infty)$.
    $X$ is {\em $\alpha - \beta$ regular} if
    for every sequence $\{ \, x_n \, \}$
     in $X$ such that $x_n \to x \in X$ and
    $\alpha(x_n, x_{n+1}) \ge 1$ and 
    $\beta(x_n, x_{n+1}) \ge 1$ for all
    $n$, then there exists a subsequence
    $\{ \, x_{n_k} \, \}$ of
    $\{ \, x_n \, \}$ such that 
    $\alpha(x_{n_k}, x_{n_{k+1}}) \ge 1$,
    $\beta(x_{n_k}, x_{n_{k+1}}) \ge 1$,
    $\alpha(x,Tx) \ge 1$, and
    $\beta(x,Tx) \ge 1$.
\end{definition}

The following is stated as Theorem~3.1 of~\cite{MasmaliEtAl}.
\begin{assert}
    \label{Masmali3.1}
    Let $(X,d,\rho)$ be a complete connected digital metric space. Let
    $T: X \to X$ and let $\alpha, \beta: X \times X \to [0.\infty)$
    be such that
    \begin{enumerate}
        \item $T$ is $\alpha - \beta$ admissible;
        \item There exists $x_0 \in X$ such that $\alpha(x_0,Tx_0) \ge 1$ and $\beta(x_0,Tx_0) \ge 1$;
        \item Either $T$ is continuous or $X$ is $(\alpha - \beta)$ regular; and
        \item For some $\psi, \varphi \in \Phi$ and all $x,y \in X$, 
              \begin{equation}
              \label{Masmail3.1eq}
              \alpha(x,Tx) \beta(y,Ty) \psi(d(Tx,Ty)) \le
               \psi(d(x,y)) - \varphi(d(x,y)).
               \end{equation}
    \end{enumerate}
    Then $T$ has a fixed point. Further, if $u$ and $v$ are fixed points
    of~$T$ such that $\alpha(u,Tu) \ge 1$, $\alpha(v,Tv) \ge 1$,
    $\beta(u,Tu) \ge 1$, and $\beta(v,Tv) \ge 1$, then $u=v$.
\end{assert}

Assertion~\ref{Masmali3.1} and its 
``proof" in~\cite{MasmaliEtAl} 
are flawed as follows. Some flaws listed below are easily
corrected but are possibly confusing.
\begin{itemize}
   \item The statement of~\cite{MasmaliEtAl} labeled (1) has an ``$x$"
          that should be ``$x_0$" - the statement should be
          \[
             x_{n+1} = Tx_n = T^{n+1}x_0 \mbox{ for all } n \ge 0.
          \]
    \item There are multiple references 
    to hypothesis (3) that should refer to hypothesis (4).
    \item Statement (3) of~\cite{MasmaliEtAl} says
          \[ d(x_{n_k}, x_{m_k}) \ge \varepsilon. ~~~~~~
             \mbox{(3) of~\cite{MasmaliEtAl}}
          \]
          It is then claimed that this implies 
          \[
             d(x_{n_k - 1}, x_{m_k}) < \varepsilon. ~~~~~~
             \mbox{(4) of~\cite{MasmaliEtAl}}
          \]
          There is no justification for this claim, 
          although it could be justified
          by choosing $n_k$ as the 
          minimal index
          of a member of the 
          subsequence for a given
          $m_k$ to satisfy (3) of~\cite{MasmaliEtAl}.
    \item Suppose~(4) of~\cite{MasmaliEtAl} is
          valid. It leads to a
          3-line chain of inequalities
          in which the second line is
          marked ``(7)". The third line
          is missing ``$\le$" at its
          beginning.
    \item The above leads to
          (9) of~\cite{MasmaliEtAl}:
    \[ \psi(d(x_{n_k},x_{m_k})) \le
    \psi(d(x_{n_k - 1},x_{m_k - 1})) -
    \phi(d(x_{n_k - 1},x_{m_k - 1}))
       \]
       which is then taken to a limit
       in order to obtain the contradiction $\varepsilon = 0$.
       But this limit depends on an
       unstated hypothesis that
       the functions $\psi$ and $\phi$
       are continuous.
   \item The contradiction mentioned 
         above resulted from assuming
         $\varepsilon > 0$. Therefore
         $\varepsilon = 0$. It follows
         that $\{x_n\}$ is a Cauchy
         sequence, and since $X$ is
         complete, $x_n \to x \in X$.
         The authors spend additional
         paragraphs to argue that
         $x$ is a fixed point of $T$;
         however, 
         $(X,d,\rho)$ is uniformly
         discrete, so $x_n = x_{n+1} = 
         Tx_n = x$
         for almost all~$n$.
    \item Since the authors' argument
          for the existence of a fixed
          point as a consequence of 
         $\{x_n\}$ being Cauchy is their
         only use of hypothesis 3)
         (recall other references to
         hypothesis~3) should be references to
         hypothesis~4)), the above shows hypothesis~3) is
         unnecessary.
\end{itemize}
Thus, the assertion as written 
is unproven. By modifying~\cite{MasmaliEtAl} as
discussed above, we get a 
proof of the following version of the
assertion.

\begin{thm}
    \label{Masmali3.1CorrectV}
    Let $(X,d,\rho)$ be a complete connected digital metric space. Let
    $T: X \to X$ and let $\alpha, \beta: X \times X \to [0.\infty)$
    be such that
    \begin{enumerate}
        \item $T$ is $\alpha - \beta$ admissible;
        \item There exists $x_0 \in X$ such that $\alpha(x_0,Tx_0) \ge 1$ and $\beta(x_0,Tx_0) \ge 1$;
 and
        \item For some continuous 
        functions $\psi, \varphi \in 
        \Phi$ and all $x,y \in X$, 
              \begin{equation}
              \label{Masmail3.1eq}
              \alpha(x,Tx) \beta(y,Ty) \psi(d(Tx,Ty)) \le
               \psi(d(x,y)) - \varphi(d(x,y)).
               \end{equation}
    \end{enumerate}
    Then $T$ has a fixed point. Further, if $u$ and $v$ are fixed points
    of~$T$ such that $\alpha(u,Tu) \ge 1$, $\alpha(v,Tv) \ge 1$,
    $\beta(u,Tu) \ge 1$, and $\beta(v,Tv) \ge 1$, then $u=v$.
\end{thm}

\subsection{``Theorem" 3.2 of~\cite{MasmaliEtAl}}
\begin{definition}
    \label{MasmaliAlphaEtcContractive}
    {\rm ~\cite{MasmaliEtAl}}
    Let $(X,d,\rho)$ be a complete digital
    metric space and let $S,T: X \to X$
    and $\alpha,\beta: X \times X \to 
    [0,\infty)$. The pair $(S,T)$ is a pair
    of {\em $\alpha - \beta - \psi - \varphi$
    contractive mappings}, where
    $\psi, \varphi \in \Phi$, 
    if
    \[ \alpha(x,Tx) \beta(y,Ty) \psi(d(Sx,Sy))
       \le \psi(d(Tx,Ty)) - \varphi(d(Tx,Ty))
        \]
    for all $x,y \in X$.
\end{definition}

\begin{remarks}
    It seems likely that there should
    be further restrictions on
    $\alpha$ and $\beta$. E.g.,
    given any pair $S,T: X \to X$
    and any $\psi,\phi$ such that
    $\psi(t) \ge \phi(t)$ for all 
    $t \ge 0$, if either of
    $\alpha$ or $\beta$ is the constant
    function with value 0, then
    according to 
Definition~\ref{MasmaliAlphaEtcContractive},
    $(S,T)$ is a pair
    of $\alpha - \beta - \psi - \varphi$
    contractive mappings.
\end{remarks}

The following is stated as Theorem 3.2 of~\cite{MasmaliEtAl}.

\begin{assert}
    \label{Masmali3.2}
    Let $(X,d,\rho)$ be a complete digital metric space. Let
    $S,T: X \to X$ be
    be $\alpha - \beta - \psi - \varphi$ 
    mappings for $\alpha, \beta: X \times X \to [0, \infty)$
    such that
    \begin{enumerate}
        \item $S(X) \subset T(X)$;
        \item $S$ is $\alpha - \beta$-admissible with respect to $T$;
        \item there exists $x_0 \in X$
        such that 
        $\alpha(Tx_0,Sx_0) \ge 1$ and
        $\beta(Tx_0,Sx_0) \ge 1$,
        and
        \item we have
              \[
                 \alpha(x,Tx) \beta(y,Ty) \psi(d(Sx,Sy)) \le 
                 \psi(d(Tx,Ty)) - \varphi(d(Tx,Ty))
              \]
        \item If $\{ \, Tx_n \, \} \subset X$ such that
              $\alpha(Tx_n, Tx_{n+1}) \ge 1$ and
              $\beta(Tx_n, Tx_{n+1}) \ge 1$ for all $n$ and
              $Tx_n \to_{n \to \infty} Tx$, then 
              there exists a subsequence 
              $\{ \,Tx_{n(k)} \, \}$ of $\{ \,Tx_n \, \}$
              and $z \in X$
              such that $\alpha(Tx_{n(k)}, Tz) \ge 1$ and
              $\beta(Tx_{n(k)}, Tz) \ge 1$ for all $k$.
        \item $T(X)$ is closed.
    \end{enumerate}
    Then $S$ and $T$ have a 
    coincidence point.
\end{assert}

Among the flaws of Assertion~\ref{Masmali3.2} as presented in~\cite{MasmaliEtAl} are the
following.
\begin{itemize}
    \item It is assumed in~\ref{Masmali3.2} (see that
          paper's Definition~2.3) that $d$ is the
          Euclidean metric, which is uniformly discrete
          on subsets of $\Z^n$. Therefore, the
          assumption of a discrete metric space need
          not be stated; the hypothesis that $T(X)$ 
          is closed, is unnecessary; and $(X,d)$ may be
          more generally assumed to be uniformly discrete.
    \item At the line marked ``(11)", there is no clear 
          justification for the claim 
          \[ x_n = Sx_n.
          \]
          This error propagates to the lines marked ``(14)".
       \item At the chain of inequalities marked ``(14)" and its
             subsequent paragraph, the error 
             noted at the line marked ``(11)" results in the conclusion 
             that $lim_{n \to \infty} d(Sx_n, Sx_{n+1})$ is 0, i.e.,
             that the line marked ``(12)" is valid.
\end{itemize}
Thus, the error noted at ``(11)" propagates 
through the ``proof." We must therefore regard
Assertion~\ref{Masmali3.2} as unproven.

\subsection{``Theorem" 3.3 of~\cite{MasmaliEtAl}}
The following is stated as Theorem~3.3
of~\cite{MasmaliEtAl}.

\begin{assert}
    \label{Masmali3.3}
    Let $(X,d,\rho)$ be a complete digital 
    metric space and
    $S,T: X \to X$ and $\alpha, \beta: X \times X \to [0, \infty)$ be mappings such that
    \begin{enumerate}
        \item $S(X) \subset T(X)$;
        \item $S$ is $\alpha - \beta$ admissible
              with respect to $T$;
        \item there exists $x_0 \in X$ such
              that $\alpha(Tx_0, Sx_0) \ge 1$
              and $\beta(Tx_0, Sx_0) \ge 1$;
        \item there exist $\psi, \varphi \in \Phi$
              such that 
              \[ \alpha(Tx,Ty) \psi(d(Sx,Sy)) \le     \psi(M(x,y)) - \varphi(M(x,y))
                            \]
              for all $x,y \in X$, where
              \[ M(x,y) = \max \{ \, d(Tx,Ty),
                 d(Tx,Sx), d(Sy,Ty), d(Sx,Ty)
                 \, \};
              \]
        \item If $\{ \, x_n \, \}$ is a sequence
              in $X$ such that 
              $\alpha(x_n,x_{n+1}) \ge 1$ and
              $\beta(x_n,x_{n+1}) \ge 1$ for all
              $n$ and 
              $Tx_n \to_{n \to \infty} Tx \in T(X)$,
              then there is a subsequence
              $\{ \, x_{n(k)} \, \}$ of 
              $\{ \, x_n \, \}$ such that
              $\alpha(Tx_{n(k)},Tz) \ge 1$ and
              $\beta(Tx_{n(k)},Tz) \ge 1$ for 
              all~$k$; and
        \item T(X) is closed.
    \end{enumerate}
    Then $S$ and $T$ have a coincidence point.
\end{assert}

The argument in~\cite{MasmaliEtAl} for
Assertion~\ref{Masmali3.3} is flawed as follows.
Statement~(14) 
of~\cite{MasmaliEtAl} 
is used to justify the 
claim that
          \[ d(Sx_n,Sx_{n+1}) \to_{n \to \infty} 0.
          \]
However, as noted above,
statement (14) is not 
correctly derived.
Thus
Assertion~\ref{Masmali3.3} is unproven.

\subsection{``Theorem" 3.4 of~\cite{MasmaliEtAl}}
The following is stated as Theorem~3.4 
of~\cite{MasmaliEtAl}.

\begin{assert}
    \label{Masmali3.4}
        Let $(X,d,\rho)$ be a complete digital 
    metric space and
    $S,T: X \to X$ and $\alpha, \beta: X \times X \to [0, \infty)$ be mappings such that
    \begin{enumerate}
        \item $S(X) \subset T(X)$;
        \item $S$ is $\alpha - \beta$ admissible
              with respect to $T$;
        \item there exists $x_0 \in X$ such
              that $\alpha(Tx_0, Sx_0) \ge 1$;
        \item there exist $\psi, \varphi \in \Phi$
              such that 
              \[ \alpha(x,Tx) \beta(y,Ty) \psi(d(Sx,Sy)) \le     \psi(M(x,y)) - \varphi(M(x,y))
                            \]
              for all $x,y \in X$, where
              \[ M(x,y) = \max \{ \, d(Sx,Sy),
                 d(Tx,Sx), d(Sy,Ty),
                 \frac{d(Sx,Ty) + d(Sy,Tx)}{2}
                 \, \};
              \]
        \item If $\{ \, x_n \, \}$ is a sequence
              in $X$ such that 
              $\alpha(x_n,x_{n+1}) \ge 1$ for all
              $n$ and 
              $Tx_n \to_{n \to \infty} Tx \in T(X)$,
              then there is a subsequence
              $\{ \, x_{n(k)} \, \}$ of 
              $\{ \, x_n \, \}$ such that
              $\alpha(Tx_{n(k)},Tz) \ge 1$ for 
              all~$k$; and
        \item T(X) is closed.
    \end{enumerate}
    Then $S$ and $T$ have a coincidence point.
\end{assert}

The argument given in~\cite{MasmaliEtAl} as
proof of Assertion~\ref{Masmali3.4} is that
it follows immediately from 
Assertion~\ref{Masmali3.2}. As we have shown
above that the latter is unproven, it follows
that Assertion~\ref{Masmali3.4} is unproven.

\subsection{``Theorem" 3.5 of~\cite{MasmaliEtAl}}
The following is stated as Theorem~3.5 
of~\cite{MasmaliEtAl}.

\begin{assert}
    \label{Masmali3.5}
    In addition to the hypotheses of
    Theorem~3.2 {\em [Assertion~\ref{Masmali3.2} 
    of the current paper]}, suppose that for
    each pair $x,y$ of common fixed points of
    $S$ and $T$ there exists $z \in X$ such
    that 
    \begin{equation}
    \label{Masmali3.5Assume}
        \mbox{$\alpha(Tx,Tz) \ge 1$,~~ 
    $\beta(Tx,Tz) \ge 1$,~~ $\alpha(Ty,Tz) \ge 1$,
    ~~ $\beta(Ty,Tz) \ge 1$,}
    \end{equation}
     and $S$ and $T$
    commute at coincidence points. Then
    $S$ and $T$ have a unique common fixed
    point.
\end{assert}

Among the flaws of the argument offered as proof 
in~\cite{MasmaliEtAl} of
Assertion~\ref{Masmali3.5}, we find
the following. The symbol ``$z$" is introduced with
          two distinct meanings: as stated above among
          the hypotheses, and then in the claim ``... there
          exists $z \in X$ such that
          $\lim_{n \to \infty} Tz_n = Tz$."
No reason is given to believe that the
latter point named ``$z$" must be the
same as the point of that name
in hypothesis~(\ref{Masmali3.5Assume}). The
argument proceeds assuming the symbol represents
both the first ``$z$" and the second ``$z$".

We conclude that Assertion~\ref{Masmali3.5} is
unproven.

\subsection{``Theorems" 4.1 and 4.2 
of~\cite{MasmaliEtAl}}
The following are stated as Theorem~4.1 and
Theorem~4.2 of~\cite{MasmaliEtAl}.

\begin{assert}
    \label{Masmali4.1}
    Let $(X,d,\rho)$ be a complete digital
    metric space. Let $A$ and $B$ be nonempty
    closed subsets of $X$. Suppose
    $\alpha: X \times X \to [0,\infty)$
    and $T: A \cup B \to A \cup B$ are such that
    \begin{enumerate}
        \item $T(A) \subset B$ and
              $T(B) \subset A$;
        \item if $\alpha(x,y) \ge 1$ then
              $\alpha(Tx,Ty) \ge 1$;
        \item there exists $x_0 \in X$
              such that $\alpha(x_0,Tx_0) \ge 1$;
        \item $T$ is continuous or $X$ is
              $\alpha$-regular;
        \item for some $\psi,\varphi \in \Phi$, 
              $\alpha(x,y) \psi(d(Tx,Ty)) \le 
               \psi(d(x,y)) - \phi(d(x,y))$ for
               all $x,y \in X$.
    \end{enumerate}
    Then $T$ has a fixed point in~$A \cap B$.
    Further if $u$ and $v$ are fixed points
    of $T$ such that $\alpha(u,Tu) \ge 1$,
    $\alpha(v,Tv) \ge 1$, $\beta(u,Tu) \ge 1$,
    and $\beta(v,Tv) \ge 1$, then $u=v$.
\end{assert}

\begin{assert}
    \label{Masmali4.2}
    Let $(X,d,\rho)$ be a complete digital
    metric space. Let $A$ and $B$ be nonempty
    closed subsets of $X$. 
    Let $Y = A \cup B$ and let $S,T: Y \to Y$ satisfy
    \begin{itemize}
        \item $T(A)$ and $T(B)$ are closed;
        \item $S(A) \subset T(B)$ and
              $S(B) \subset T(A)$;
        \item $T$ is one-to-one;
        \item for some $\psi, \varphi \in \Phi$,
          \[ \psi(d(Sx,Sy)) \le 
              \psi(M(x,y)) - \varphi(M(x,y))~~~\mbox{for all }
              x,y \in A \times B;          
          \]
          {\em [presumably, the latter qualification is
          meant to be $x,y \in A \cup B$]}, where
          \[ M(x,y) = \max \{ \, d(Tx,Ty), d(Tx,Sx), d(Sy,Ty), d(Sx,Ty)
             \, \}.    
          \]
          Then $S$ and $T$ have a coincidence point in $A \cap B$.
          Further, if $S$ and $T$ commute at their coincidence point, then $S$ and $T$ have a unique common fixed point in 
          $A \cap B$.
    \end{itemize}
\end{assert}

It seems likely that there should be a hypothesis
          that $A \cap B \neq \emptyset$ in both
          Assertion~\ref{Masmali4.1} and
          Assertion~\ref{Masmali4.2}.
More definitely, the presentations of these 
assertions are flawed as follows.
The argument offered
          in~\cite{MasmaliEtAl} for
          Assertion~\ref{Masmali4.1} depends
          on Assertion~\ref{Masmali3.1}, and
          the argument offered
          in~\cite{MasmaliEtAl} for
          Assertion~\ref{Masmali4.2}
          depends on Assertion~\ref{Masmali3.2}.
Since we have shown above that 
Assertions~\ref{Masmali3.1} and~\ref{Masmali3.2}
are unproven, it follows
that Assertions~\ref{Masmali4.1} 
and~\ref{Masmali4.2} are unproven.

\subsection{Corollaries 4.3 and 4.4 of~\cite{MasmaliEtAl}}
The following Assertions~\ref{Masmali4.3} 
and~\ref{Masmali4.4} are stated in~\cite{MasmaliEtAl}
as Corollaries~4.3 and~4.4, respectively,
as immediate consequences of Assertion~\ref{Masmali4.2}.
Since we have shown above that
Assertion~\ref{Masmali4.2} is unproven,
it follows that Assertions~\ref{Masmali4.3} 
and~\ref{Masmali4.4} are unproven.

\begin{assert}
    \label{Masmali4.3}
    Let $(X,d,\rho)$ be a complete digital metric space.
    Let $A$ and $B$ be nonempty closed subsets of $X$.
    Let $Y = A \cup B$ and $S, T: Y \to Y$, satisfying the following. 
    \begin{itemize}
        \item $T(A)$ and $T(B)$ are closed.
        \item $S(A) \subset T(B)$ and $S(B) \subset T(A)$.
        \item $T$ is one-to-one.
        \item There exist $\psi, \varphi \in \Phi$ such that
            \[ \psi(d(Sx,Sy)) \le \psi(M(x,y)) - \phi(M(x,y))~~~
                \mbox{for all } x,y \in Y,
            \]
            where
            \[ M(x,y) = \max \{ \, d(Sx,Sy), d(Tx,Sx), d(Sy,Ty),
               \frac{d(Sx,Ty) + d(Sy,Tx)}{2} \, \}.
            \]
            Then $S$ and $T$ have a coincidence point in $A \cap B$.
            Further, if $S$ and $T$ commute at their coincidence point, then $S$ and $T$ have a unique common fixed point in
            $A \cap B$.
    \end{itemize}
\end{assert}

\begin{assert}
    \label{Masmali4.4}
        Let $(X,d,\rho)$ be a complete digital metric space.
    Let $A$ and $B$ be nonempty closed subsets of $X$.
    Let $Y = A \cup B$ and $S, T: Y \to Y$, satisfying the following. 
    \begin{itemize}
        \item $T(A)$ and $T(B)$ are closed.
        \item $S(A) \subset T(B)$ and $S(B) \subset T(A)$.
        \item $T$ is one-to-one.
        \item There exist $\psi, \varphi \in \Phi$ such that
            \[
            \psi(d(Sx,Sy)) \le \psi(d(Tx,Ty)) - \varphi(d(Tx,Ty))~~~
                \mbox{for all } x,y \in Y.
            \]
            Then $S$ and $T$ have a coincidence point
            in $A \cap B$. Further, if $S$ and $T$ 
            commute at their coincidence point, then
            $S$ and $T$ have a unique common fixed 
            point in $A \cap B$.
    \end{itemize}
\end{assert}

\section{On~\cite{Shukla}}
\subsection{``Theorem" 3.1}
The following is stated as Theorem~3.1 of~\cite{Shukla}.

\begin{assert}
\label{Shukla3.1}
Let $(X,d, \kappa)$ be a digital metric space and $S,T$ be self maps on $X$ 
satisfying
\[ d(Sx,Ty) \le \alpha d(x, Sx) + d(y,Ty) \mbox{ for all }
   x, y \in X \mbox{ and } 0 < \alpha < 1/2.
\]
Then $S$ and $T$ have a unique common fixed point in $X$.
\end{assert}

That this assertion is incorrect is shown by the following.

\begin{exl}
Let $D$ be the ``diamond," 
\[ D = \{ \, (1,0), (0,1), (-1,0), (0,-1) \, \},
\]
a $c_2$-digital simple closed curve. 

Let $S = \id_D$ and let $T(x) = -x$.
If $d$ is the Euclidean metric, we have
\[ d(Sx,Ty) \le diameter(D) = 2 = \alpha (0) + 2 =
    \alpha d(x,Sx) + d(y,Ty). 
\]
Since $T$ has no fixed point in $D$, we have established that
Assertion~\ref{Shukla3.1} is not generally valid.
\end{exl}

\subsection{``Theorem" 3.2}
The following is stated as Theorem~3.2 of~\cite{Shukla}.

\begin{assert}
\label{Shukla3.2}
    Let $(X,d,\kappa)$ be a digital metric space and let
    $S,T: X \to X$ such that
    \[ d(Sx,Ty) \le \alpha d(x,Ty) + d(y,Sx) 
    \]
    for all $x,y \in X$ and $0 < \alpha < 1/2$.
    Then, $S$ and $T$ have a unique common fixed point in $X$.
\end{assert}

There are multiple errors in the argument of~\cite{Shukla} offered
as proof of Assertion~\ref{Shukla3.1}. We quote one section of this argument,
with labeled lines. \newline \newline
Begin quote:
\begin{quote}
Consider
\begin{equation}
\label{a}
d(x_1,x_2) = d(Sx_0,Tx_1)
\end{equation}
Now
\begin{equation}
\label{b}
d(x_1,x_2) \le \alpha d(x_0,Tx_1) + d(x_1,Sx_0)
\end{equation}
\begin{equation}
\label{c}
d(x_1,x_2) \le \alpha d(x_0,x_2) + d(x_1,x_1)
\end{equation}
\begin{equation}
\label{d}
d(x_1,x_2) \le \alpha d(x_0,x_1) + d(x_1,x_2)
\end{equation}
\begin{equation}
\label{e}
d(x_1,x_2) - \alpha d(x_1,x_2) \le \alpha d(x_0,x_1)
\end{equation}
\begin{equation}
\label{f}
(1 - \alpha)d(x_1,x_2) \le \alpha d(x_0,x_1)
\end{equation}
\end{quote}
End quote \newline

Statement~(\ref{d}) does not appear
correctly derived from~(\ref{c}), 
and is equivalent to the unhelpful
    \[ 0 \le \alpha d(x_0,x_1)
    \]

Clearly,~(\ref{e}) is not
          correctly derived
          from~(\ref{d}).
          
Near the bottom of page 229
of~\cite{Shukla}, we see
          \begin{quote}
              $\ldots$ there is a point $u \in X$ such that $x_n \to u$.
              Therefore, subsequence $<Sx_{2n}> \to u \ldots$
              and $<T_{2n+1}> \to u$
              since $S$ and $T$ are
              $(\kappa,\kappa)$-continuous ... we have, $Su = u$ .... 
          \end{quote}
          This is incorrect reasoning. The author does not say if the undefined term
          ``map" is assumed to include a hypothesis of digital continuity; nor
          is continuity proven. Further, 
          Example~\ref{nonconverge}, given below, provides 
          a counterexample to the claim that such subsequences imply
          the existence of a fixed point, even if $S$ is digitally continuous.
          Thus, the existence of a fixed point is not established.

In statements~(1) and~(2) on page 230, there are instances of ``="
          that should be ``$\le$".
          
Even if we assume common fixed points $u,v$ of $S$ and $T$, the
argument provided on page 230 to show $u$ and $v$ coincide is incorrect. The
author claims the inequality $d(u,v) - d(v,u) \le 0$ implies
$d(u,v) = 0$, but clearly it does not.

Thus, we conclude that Assertion~\ref{Shukla3.1} is unproven.

\begin{exl}
    \label{nonconverge}
    Let $S = T: \N \to \N$ be the function 
    \[ S(n) = T(n) = \left \{ \begin{array}{ll}
        0 & \mbox{if $n=1$;} \\
        1 & \mbox{if $n \neq 1$}.
    \end{array} \right .  
    \]
    Let $x_n = n$. $S$ and $T$ are both
    $c_1$-continuous. Also,
    $Sx_{2n} \to 1$
    and $Tx_{2n+1} \to 1$ but
    $S(1) \neq 1$. Further, neither
    $S$ nor $T$ has a fixed point.
    \end{exl}

\subsection{``Theorem" 3.3}
The following is stated as Theorem~3.3 in~\cite{Shukla}.
\begin{assert}
    \label{Shukla3.3}
    Let $(X,d,\kappa)$ be a digital metric space and let $S,T$ be
    self-maps on $X$ such that
    \[
    d(Sx,Ty) \le a d(x,Sx) + b d(y,Ty) + c d(x,y) \mbox{ for all } x,y \in X,
    \]
    where $a,b,c$ are nonnegative real numbers such that
    $a + b + c < 1$. Then $S$ and $T$ have a unique common fixed point in $X$.
\end{assert}

The argument offered in~\cite{Shukla} as proof of this assertion is flawed
as follows.
\begin{itemize}
    \item Errors similar to above for Assertion~\ref{Shukla3.2}:
          it is incorrectly claimed that (unestablished 
          digital) continuity and
          $x_n \to u$ imply $Sx_{2n} \to u$ and $Tx_{2n+1} \to u$. From this
          is wrongly (see Example~\ref{nonconverge}) concluded that $u$ 
          is a common fixed point of $S$ and $T$.
    \item The argument for uniqueness of any fixed point has a correctible
          error:
          \[
            ``d(u,v) = cd(u,v)" \mbox{ should be } ``d(u,v) \le cd(u,v)",
          \]
          which does imply the desired conclusion, $d(u,v)=0$.
\end{itemize}
Due to the errors discussed in the first bullet,
Assertion~\ref{Shukla3.3} is unproven.

\subsection{``Theorems" 3.4 and 3.5}
We discuss two assertions of~\cite{Shukla}
whose arguments given as proofs contain the same
errors.

The following is stated in~\cite{Shukla} as Theorem~3.4.

\begin{assert}
    \label{Shukla3.4}
    Let $(X,d,\kappa)$ be a digital metric space and let $S$ be a
    self-map on $X$ such that
    \[
    d(Sx,Sy) \le \frac{[a d(y,Sy)][1+d(x,Sx)]}{1+d(x,y)} + b d(x,y)
    \]
    for all $x,y \in X$, where
    $a,b \ge 0$ and $a+b < 1$.
    Then $S$ has a unique fixed point in $X$.
\end{assert}

The following is stated in~\cite{Shukla} as Theorem 3.5.

\begin{assert}
    \label{Shukla3.5}
    Let $(X,d,\kappa)$ be a digital metric space and let $S$ be a
    self-map on $X$ such that
    \[
    d(Sx,Sy) \le \frac{[a d(y,Sy)][1+d(x,Sx)]}{1+d(x,y)} + b d(x,y) +
                  c \frac{d(y,Sy)+d(y,Sx)}{d(y,Sy)d(y,Sx)}
    \]
    for all $x,y \in X$, where $a,b,c$ are nonnegative and $a+b+c < 1$.
    Then $S$ has a unique fixed point in $X$.
\end{assert}

The arguments offered in~\cite{Shukla} for these
assertions are both flawed as follows.
\begin{itemize}
   \item Errors similar to above for Assertions~\ref{Shukla3.2}
          and~\ref{Shukla3.3}:
          it is incorrectly claimed that (questionable 
          digital) continuity and
          $x_n \to u$ imply $Sx_{2n} \to u$ and $Sx_{2n+1} \to u$. From this
          is wrongly (see Example~\ref{nonconverge}) concluded that $u$ 
          is a fixed point of $S$.
    \item Despite the author's claim, the statement 
          $d(u,u)=0$ fails to establish uniqueness of a 
          fixed point of $S$; one must show $d(u,v)=0$
          for $u,v \in \Fix(S)$.
\end{itemize}
Thus, Assertions~\ref{Shukla3.4} and~\ref{Shukla3.5} 
are unproven.

\section{Further remarks}
We have continued the work
of~\cite{BxSt19,Bx19,Bx19-3,Bx20,Bx22}
in discussing flaws in papers rooted
in the notion of a digital metric space.
The papers we have considered have many
errors and assertions that turn out to be
trivial.

Although authors are responsible for
their errors and other shortcomings,
it is clear that many of the papers
studied in the current paper were
given inadequate review.

\end{document}